\documentclass{amsart}

\usepackage{todonotes}
\usepackage{hyperref}
\usepackage{amsmath,amssymb,tikz}
\usetikzlibrary{arrows,shapes,positioning,decorations.markings}
\usepackage{verbatim}

\let\emph\relax 
\DeclareTextFontCommand{\emph}{\bfseries\em}

\newtheorem{theorem}{Theorem}[section]
\newtheorem{lemma}[theorem]{Lemma}

\renewcommand{\wr}{\mathop{\mathrm{wr}}}

\newcommand{\Sym}{\mathop{\mathrm{Sym}}}
\newcommand{\Cay}{\mathop{\mathrm{Cay}}}
\newcommand{\Aut}{\mathop{\mathrm{Aut}}}

\newcommand{\CM}{\mathop{\mathrm{CM}}}
\newcommand{\rr}{\mathfrak{r}}
\numberwithin{equation}{section}

\def\cent#1#2{{\bf C}_{#1}(#2)}

\begin{document}
\title[Mapical regular representations]{Almost all cayley maps are mapical regular representations}

\author[P. Spiga]{Pablo Spiga}
\address{Pablo Spiga,
Dipartimento di Matematica e Applicazioni, University of Milano-Bicocca,\newline
Via Cozzi 55, 20125 Milano, Italy}\email{pablo.spiga@unimib.it}

\author[D. Sterzi]{Dario Sterzi}
\address{Dario Sterzi,
Dipartimento di Matematica \lq\lq Tullio Levi-Civita\rq\rq,\newline
 University of Padova, Via Trieste 53, 35121 Padova, Italy
}\email{dario.sterzi@studenti.unipd.it}

\begin{abstract}
Cayley maps are combinatorial structures built upon Cayley graphs on a group. As such the original group embeds in their group of automorphisms, and one can ask in which situation the two coincide (one then calls the Cayley map  a mapical regular representation or MRR) and with what probability. The first question was answered by Jajcay. In this paper we tackle the probabilistic version, and prove that as groups get larger the proportion of MRRs among all Cayley Maps approaches 1.
\end{abstract}

\maketitle
\section{Introduction}\label{intro}
In this first section we define Cayley graphs and maps, give some context and state our main theorem. In the second section we prove the theorem.
\subsection{Cayley graphs}
We consider only finite groups and finite graphs in this paper. As usual a \textbf{\textit{graph}} $\Gamma$ is an ordered pair $(V, E)$ with $V$ a finite non-empty set and with $E$ a collection of $2$-subsets of $V$.  An \textit{\textbf{automorphism}} of a graph is a permutation on $V$ that preserves the set $E$, and a \emph{path} on a graph is a sequence $v_1,v_2,\dots, v_n$ of adjacent vertices, i.e. $\{v_i,v_{i+1}\}\in E $ for all $i$. The \textbf{\textit{neighbourhood}} of a vertex $v$ is the set $\Gamma (v)=\{w\in V| \{v,w\}\in E\}$ of all vertices connected to it by an edge.

Let $R$ be a group and let $S$ be an inverse-closed subset of $R$, that is, $S=\{s^{-1}\mid s\in S\}$. The \textit{\textbf{Cayley graph}} $\mathrm{Cay}(R,S)$ is the graph with $V=R$ and with $\{r,t\}\in E$ if and only if $tr^{-1}\in S$, i.e. $E=\{\{r,sr\}| s\in S, r\in R\}$. The condition $S=S^{-1}$ is imposed to guarantee that $tr^{-1}\in S$ if and only if $rt^{-1}\in S$.
A path $r_0,r_1 ,\dots,r_n$ in a Cayley graph can be specified equivalently by its starting vertex $r_0$ together with the unique sequence of elements $s_1,s_2,
\dots , s_n$ from $S$ such that $r_{i+1}=s_{i+1}r_i$.
Usually one is interested in connected Cayley graphs, where for any two vertices there is at least one path connecting them. This is equivalent to the requirement that $S$ is a set of generators for the group. We shall assume so throughout this paper.

A \textit{\textbf{graphical regular representation}} (GRR) for a group $R$ is a graph whose automorphism group is the group $R$ acting regularly on the vertices of the graph. (A permutation group $R$ is \textit{\textbf{regular}} if it is transitive and if the identity element of $R$ is the only element fixing some point of the domain.) It is an easy observation that the right regular action of  $R$ on itself preserves the edges, so $R$ embeds in $\Aut(\mathrm{Cay}(R,S))$.\footnote{We let automorphisms act on the right, so we will write $x^\varphi$ to denote the image of  the vertex $x$ under the automorphism $\varphi$, and we shall take $x^{\varphi \psi}$ to mean $(x^\varphi)^\psi$.} A GRR for $R$ is therefore a Cayley graph on $R$ that admits no other automorphisms.

The main thrust of much of the work through the 1970s was to determine which groups admit GRRs. This question was ultimately answered by Godsil in~\cite{8}. It was conjectured by Babai and Godsil  that, except for two natural families of groups,  GRRs not only exist, but they are abundant, that is, with probability tending to $1$ as $|R|\to\infty$, a Cayley graph on $R$ is a GRR.  The first author reported the recent progress in~\cite{spiga1,spiga2,spiga3,spiga4} on the Babai-Godsil conjecture at the SIGMAP 2022 conference at the University of Alaska Fairbanks. During this conference, Robert Jajcay has suggested  a similar investigation for Cayley maps.\footnote{During the preparation of this paper, Xia and Zheng have announced a solution to the Babai-Godsil conjecture, see~\cite{xia2022}.}
We now give some background on Cayley maps, state Jajcay's question and state our main result.

\subsection{Graph maps and Cayley maps}
Let $\Gamma:=(V,E)$ be a graph. Given $v\in V,$ we let $\Gamma(v)$ denote the neighbourhood of $v$ in $\Gamma$. A \textit{\textbf{rotation}} on $\Gamma$ is a set $\rho:=(\rho_v)_{v\in V}$, where each $\rho_v:\Gamma(v)\to \Gamma(v)$ is a cyclic ordering \footnote{A cyclic ordering on a (finite) set is a permutation with no fixed points and a single cycle in its cycle decomposition. } of $\Gamma(v)$. A \textit{\textbf{map}} is a pair $(\Gamma,\rho)$, where $\Gamma$ is a connected graph and $\rho$ is a rotation of $\Gamma$.

The idea behind maps is that they represent a CW complex structure on an orientable surface whose 1-skeleton is the given graph, essentially an embedding of the graph in an orientable surface disconnecting it into disks. See for instance \cite{Gross} for  details. The $\rho_v$ are the cyclic orderings of the edges incident to $v$ in the embedding. 

Intuitively, an automorphism of a map $(\Gamma,\rho)$ is a pair: an automorphism of the graph and an  oriented homeomorphism of the surface that are compatible trough the embedding. 
Combinatorially this translates to an automorphism of $\Gamma$ (a permutation of the vertices preserving the edges) which also preserves the rotation $\rho$. 
In order to make this idea precise, we make a slight detour. 
Let $\mathrm{Aut}(\Gamma)$ be the automorphism group of $\Gamma$ and let $R(\Gamma)$ be the collection of all rotations of $\Gamma$. 
Now, $\mathrm{Aut}(\Gamma)$ has a natural action on $R(\Gamma)$:
\begin{alignat*}{2}
R(\Gamma)\times\mathrm{Aut}(\Gamma)&\longrightarrow&R(\Gamma)\\
(\rho,\varphi)&\longmapsto &\rho^{(\varphi)},
\end{alignat*}
where $\rho_{v^\varphi}^{(\varphi)}=\varphi^{-1}\rho_v\varphi$, for all $v\in V$. In other words, the rotation $\rho^{(\varphi)}$ at the vertex $v^\varphi$ takes $u^\varphi$ to $w^\varphi$ when $\rho_v$ takes $u$ to $w$.
Now, an \textit{\textbf{automorphism}} of a map $M=(\Gamma,\rho)$ is an automorphism $\varphi$ of the graph $\Gamma$ such that $\rho^{(\varphi)}=\rho$, that is, 
\begin{equation}\label{eq:mapaut}
  \rho_{v^\varphi}=\varphi^{-1}\rho_v\varphi, \ \text{ for each $v$ vertex of } \Gamma.
\end{equation} 
It is well known~\cite{Biggs} that, if the underlying graph is connected, a map automorphism is determined uniquely by its value on an oriented edge (i.e. an ordered pair of adjacent vertices). 
We recall briefly the reason: suppose $\varphi$ is a map automorphism, $w_0,w_1$ are adjacent vertices mapped to $w_0^\varphi$ and $w_1^\varphi$ respectively and $w_0,w_1,\dots,w_t$ is a path in the graph. We can describe the path as a sequence of left and right turns, or with a closer analogy as the exits to take at consecutive roundabouts.
There must be natural numbers $n_i$ for $i \in \{1,\dots,t-1\}$ such that $w_{i+1}=w_{i-1}^{\rho_{w_i}^{n_i}}$. 
Thus the path $\varphi(w_0),\varphi(w_1),\dots,\varphi(w_t)$ is uniquely determined  by the relations
\begin{equation*}
    w_{i+1}^\varphi=w_{i-1}^{\varphi\rho_{w_i^\varphi}^{n_i}}
    \text{ for }i \in \{1,\dots,t-1\}\text{.}
\end{equation*}
In other words the automorphism group of a map on a connected graph acts semiregularly on the set of oriented edges.

Let now $R$ be a group and $S$ as above an inverse-closed set of generators excluding the identity.
For every cyclic ordering $\mathfrak{r}:S\to S$, we define the \textit{\textbf{Cayley map}} 
$CM(R,S,\mathfrak{r})=(\Gamma,\rho)$ as follows: $\Gamma$ is the Cayley graph $\mathrm{Cay}(R,S)$ and, for every $g\in R$ and for every $x$ lying in the neighbourhood $\Gamma(g)$ of the vertex $g$,
\begin{alignat*}{3}\label{def:cayleyrotation}
\rho_g: & \Gamma(g)&\longrightarrow&\Gamma(g)\\
&x&\longmapsto &\rho_g(x):=g\mathfrak{r}(g^{-1}x).
\end{alignat*}
This is the unique map with the prescribed rotation $\mathfrak{r}$ around the identity vertex $e\in R$ such that the right regular action of the group on the Cayley graph preserves the rotation.

Combinatorially, we may think of a Cayley map as just a triple $(R,S,\mathfrak{r})$, where 
\begin{itemize}
\item $R$ is a finite group, 
\item $S\subseteq R\setminus\{e\}$ is a generating set with $S=S^{-1}$, and
\item $\mathfrak{r}:S\to S$ is a cyclic ordering.
\end{itemize}

\subsection{Mapical regular representations and the question of Jajcay}
Given a Cayley map $CM(R,S,\rho)$, the right regular representation of $R$ is contained in the automorphism group of $CM(R,S,\mathfrak{r})$. Analogously to GRRs, we say that $CM(R,S,\mathfrak{r})$ is a \textit{\textbf{mapical regular representation}} (or MRR for short) if  
$$\mathrm{Aut}(CM(R,S,\mathfrak{r}))\cong R.$$
As far as we are aware, this definition was coined by Robert Jajcay in~\cite{Jajcay}. Theorem~7 in~\cite{Jajcay} shows that each finite group not isomorphic to $\mathbb{Z}_3$ or $\mathbb{Z}_2^2$ possesses an MRR. Observe that $CM(R,S,\mathfrak{r})$ is a MRR if and only if the only automorphism of $CM(R,S,\mathfrak{r})$ fixing a vertex is the identity.

Once that the existence of MRRs is established it is fairly natural to investigate the abundance of MRRs among Cayley maps. Indeed, Robert Jajcay has asked whether, as $|R|\to\infty$, the proportion of MRRs among Cayley maps on $R$ tends to $1$. There are various ways to approach this counting problem and in this paper we are only concerned with \textit{\textbf{labelled}} Cayley maps, where two Cayley maps $CM(R,S,\mathfrak{r})$ and $CM(R,S',\mathfrak{r}')$ over the same group are considered to be the same if and only if $S=S'$ and $\mathfrak{r}=\mathfrak{r}'$. We managed to answer the question positively.

\begin{theorem}\label{thm:main}
  As $|R|\to\infty$, the proportion of MRRs among labelled Cayley maps on $R$ tends to $1$.
\end{theorem}

Xia and Zheng~\cite{xia2022} have recently announced a positive solution of the Babai-Godsil conjecture. This means that, except for abelian groups of exponent greater than $2$ and for generalized dicyclic groups, with probability tending to $1$ as $|R|\to\infty$, a Cayley graph on $R$ is a GRR. There are some relations between our work and the work in~\cite{xia2022}, for instance, both results depend upon a theorem on group generation due to Lubotzky~\cite{Lub}. However, there is no direct implication between our Theorem~\ref{thm:main} and the main result in~\cite{xia2022}; for instance, a positive solution of the Babai-Godsil conjecture does not imply the veracity of Theorem~\ref{thm:main}. Indeed, the number of Cayley maps on a fixed Cayley graph $\Cay(R,S)$ is $(|S|-1)!$, thus most Cayley maps have almost all the group as connection set of the underlying Cayley graph, while a random Cayley graph has roughly $|R|/2$ elements in its connection set. 
More precisely: the two questions consider different marginal probability distributions on the space of Cayley graphs.

\section{proof of main theorem}
In this section, we let $R$ be a finite group and we let $r$ denote its order.

We explore the inclusions $R \leq \Aut(\CM(R,S,\mathfrak{r}))\leq \Sym(R)$. 
Our strategy is proving a necessary condition for intermediate subgroups between $R$ and $\Sym(R)$ to be automorphism groups of Cayley maps, bound the number of subgroups satisfying this condition and then bound the number of pairs $(S,\mathfrak{r} )$ compatible with at least one of them. 

The following lemma is essentially a restatement of insights in~\cite{Jajcay93}.
\begin{lemma}\label{Necessary condition}
  For any  Cayley map $CM(R,S,\mathfrak{r})$,  the stabilizer $\Aut(CM(R,S,\mathfrak{r}))_e$ of the identity vertex $e$  is  cyclic of order less than $|R|$. If $\Aut(CM(R,S,\mathfrak{r}))_e=\langle \gamma \rangle$, then $S^\gamma=S$ and the restriction $\gamma|_S$ has the same order as $\gamma$ and it is a power of $\rr$.
\end{lemma}
\begin{proof}
  An automorphism fixing $e$ sends its neighbourhood $\Gamma(e)=S$ to itself.
  
  Since the action on oriented edges is semiregular, an element of the stabilizer is uniquely determined by its action on $S$, i.e. the restriction mapping 
  \begin{alignat*}{2}
    \Aut(CM(R,S,\mathfrak{r}))_e&\longrightarrow&\Sym(S)\\
     \varphi&\longmapsto& \varphi|_S
  \end{alignat*} 
  is injective.

  Moreover, if $\varphi\in\Aut(CM(R,S,\mathfrak{r}))_e$, then from~\eqref{eq:mapaut} we have $\mathfrak{r}=\varphi^{-1}\mathfrak{r}\varphi$, i.e., $\varphi|_S\in \cent{\Sym(S)}{\rr}$.
  From standard computations in permutation groups, we have $\cent{\Sym(S)}{\rr}=\langle \rr\rangle$. Thus $\Aut(CM(R,S,\rr))_e$ is isomorphic to a subgroup of a cyclic group, hence $\Aut(CM(R,S,\rr))_e$ is cyclic and all its elements restricted to $S$ are powers of $\rr$.
\end{proof}
Until now, we have adopted the view that a group $R$ with $r$ elements can be embedded into $\Sym(r)$ using the usual  right regular representation. It is convenient for our exposition to consider the equivalent formulation ``$R$ is a regular subgroup of $\Sym(r)$'', here regular means that for any two points in $\{1,\dots , r\}$ there exists a unique permutation in $R$ sending the first to the second. 

\begin{lemma}\label{lemma2}
  For  every regular subgroup $R$ of $\mathrm{Sym}(r)$, the number of subgroups $G$ of $\mathrm{Sym}(r)$ with
  \begin{itemize}
    \item $R<G$ and
    \item $G_1$  cyclic  and $|G_1|\le r-1$  (where $G_1$ is the stabiliser of $1$ in $G$)
  \end{itemize} is at most $2^{7(\log_2r)^2+12\log_2r}$.
  \end{lemma}
\begin{proof}
  Given $G$ and $G'$ two  abstract groups and $H\leq G$, $H'\leq G'$, we write $(G,H)\sim (G',H')$ if there exists a group isomorphism $\phi:G\to G'$ with $H^\phi=H'$. Clearly, $\sim$ defines an equivalence relation. We denote by $[G,H]$ the $\sim$-equivalence class containing $(G,H)$. Now consider 
\begin{equation*}
\mathcal{M}=\{[G,H]\mid G \textrm{ is }(\log_2r+1)\textrm{-generated}, H\leq G, |G|\leq r(r-1), \textrm{ and }H \textrm{ is cyclic}\}.
\end{equation*} 

\smallskip

\noindent\textsc{Claim 1: }We have
\begin{equation}\label{eq2}
|\mathcal{M}|\leq 2^{4(\log_2(r))^2+12\log_2r}.
\end{equation}

\smallskip

  \noindent  From~\cite[Theorem~$1$]{Lub}  together with~\cite[Remark~$3$(1)]{Lub} we get that the number of isomorphism classes of groups of order $N$ that are $d$-generated is at most $N^{2(d+1)\log_2(N)}=2^{2(d+1)(\log_2(|N|))^2}$. In particular, applying this theorem with $d:=\log_2(r)+1$ and with $N\leq r(r-1)$, we get that the number of groups $G$ that are $(\log_2(r)+1)$-generated and of order at most $r(r-1)$ is at most $2^{4(\log_2(r)+2)\log_2r}\cdot r^2$ (observe that the second factor counts the number of choices for $N$: the cardinality of $G$).
Now, let $G$ be a group of order at most $r(r-1)$. Since $G$ has at most $|G|<r^2$ cyclic subgroups $H$, our claim is proved.~$_\blacksquare$

\smallskip

Now, let $R$ be a regular subgroup of $\mathrm{Sym}(r)$ and let $\mathcal{S}_R$ be the set of subgroups $G$ of $\mathrm{Sym}(r)$ with $R<G$, with $G_1$ cyclic and with $|G_1|\le r-1$. Since $G=RG_1$ and since $R$, as any group of order $r$,  needs at most $\log_2 r$ generators, we deduce that $G$ needs at most $\log_2(r)+1$ generators.  

 \smallskip
 
  \noindent\textsc{Claim 2:}  We have 
\begin{equation}\label{eq1}
|\mathcal{S}_R|\leq 2^{3(\log_2r)^2}|\mathcal{M}|.
\end{equation}

\smallskip

\noindent  Every $G\in\mathcal{S}_R$  determines an element of $\mathcal{M}$ via the mapping $\varphi:G\mapsto [G,G_1]$. 

We show that there are at most $2^{3(\log_2r)^2}$ elements of $\mathcal{S}_R$ having the same image via $\varphi$, from which~\eqref{eq1} immediately follows. We argue by contradiction and we let $G^1,\ldots,G^\ell\in\mathcal{S}_R$ with $\varphi(G^i)=\varphi(G^1)$, for every $i\in \{1,\ldots,\ell\}$, where $\ell>2^{3(\log_2r)^2}$. Thus there exists a group isomorphism $\phi_i:G^1\to G^i$ with $(G^i)_1=((G^1)_1^{\phi_i}$. Therefore the permutation representation of $G^1$ on the coset space $G^1/(G^1)_1$  is permutation isomorphic to the permutation representation of $G^i$ on the coset space $G^i/(G^i)_1$. Thus $G^1$ and $G^i$ are conjugate via an element of $\mathrm{Sym}(r)$, that is, $G^1=(G^i)^{\sigma_i}$ for some $\sigma_i\in \mathrm{Sym}(r)$. Now, as $G^1$ acts transitively on $\{1,\ldots,r\}$, replacing $\sigma_i$ by an element of the form $g_i\sigma_i$ (for some $g_i\in G^1$), we may assume that $\sigma_i$ fixes $1$, that is, $1^{\sigma_i}=1$.

As $R\leq G^i$ for every $i$, we get that $R^{\sigma_1},\ldots,R^{\sigma_\ell}$ are $\ell$ regular subgroups of $G^1$. 
Since $R$ is $\log_2(r)$-generated, we see that $G^1$ contains at most $|G^1|^{\log_2(r)}\leq r^{2\log_2 r}=2^{2(\log_2r)^2}$ distinct subgroups of order $r$. 
In particular, since $\ell>2^{3(\log_2r)^2}$, we see that $R^{\sigma_{i_1}}=\cdots=R^{\sigma_{i_t}}$ for some $t>2^{(\log_2(r))^2}$ and some subset $\{i_1,\ldots,i_t\}$ of size $t$ of $\{1,\ldots,\ell\}$. Therefore $\sigma_{i_1}\sigma_{i_j}^{-1}$ normalises $R$. 
As $1^{\sigma_{i_1}\sigma_{i_j}^{-1}}=1$, $\sigma_{i_1}\sigma_{i_j}^{-1}$ is an automorphism of $R$, for every $j\in \{1,\ldots,t\}$. Since $R$ has at most $|R|^{\log_2(r)}=2^{(\log_2(r))^2}$ automorphisms, we get $\sigma_{i_1}\sigma_{i_j}^{-1}=\sigma_{i_1}\sigma_{i_{j'}}^{-1}$ for two distinct indices $j$ and $j'$. Thus $\sigma_{i_j}=\sigma_{i_{j'}}$ and $G^{i_j}=(G^1)^{\sigma_{i_j}^{-1}}=(G^1)^{\sigma_{i_{j'}}^{-1}}=G^{i_{j'}}$, which is a contradiction.~$_\blacksquare$

\smallskip

From~~\eqref{eq2} and~\eqref{eq1}, we have $$|\mathcal{S}_R|\le 2^{7(\log_2 r)^2+12\log_2r},$$ and the proof of this lemma immediately follows.
\end{proof}
It remains to estimate the number Cayley maps on a group $R$ compatible with a fixed intermediate subgroup $G$ with cyclic point stabilizer $H$.
\begin{lemma}\label{lemma:compatible mappings}
  For every pair of subgroups $R$ and $H$ of $\Sym(r)$ such that $R$ is regular and $H= \langle \gamma \rangle$ is non-identity, cyclic of order less than $r$ and fixing the point $1$, let $$
  \mathcal{R_\gamma}=
  \{(S,\rr)|
  S\subseteq\{2,\dots,r\},\,\rr \text{ cyclic ordering on }S,\, \gamma \in \Aut(CM(R,S,\rr))\}$$ be the set of all Cayley maps on $R$ admitting $\gamma$ as an automorphism. Then $|\mathcal{R}_\gamma|\le
   (r-1)\frac{r}{2} \left\lfloor r/2\right \rfloor! 2^{r}$.
\end{lemma}
\begin{proof}
Let $l$ be the order of $\gamma$.
  From Lemma \ref{Necessary condition}, if  $(S,\rr) \in \mathcal{R}_\gamma$, then  $S^\gamma=S$; thus $S$ is a union of $H $-orbits. 
  Moreover, $\gamma|_S$ is a power of $\rr$;  hence  $\gamma|_S$ is a product of $k$  disjoint cycles all of the same length $l$ fixing no point in $S$. Clearly $kl = |S| <r$. For a fixed $S$ (and hence $k$ and $l$),  $\rr\in \cent{\Sym(S)}{\gamma|_S}$. From routine computations,  $\cent{\Sym(S)}{\gamma|_S}$ is isomorphic to the wreath product $C_l \wr \Sym(k)$. Thus, given $S$, there are at most $l^kk!$ choices for $\rr$.
  
  If $n_l$ is the number of cycles of length $l$ in the cycle decomposition of $\gamma$, then there are ${n_l\choose l}$ choices for $S$ such that $\gamma|_S$ decomposes in $k$ cycles of length $l$. 

  Putting everything together, we have
  \begin{align}\label{Dario}
    |\mathcal{R}_\gamma|
    &\leq \sum_{l=2}^{r-1}\sum_{k=1}^{n_l} \binom{n_l}{k}k!l^k.
  \end{align}
  Of course $ln_l \leq |S|<r$ and hence $n_l<r/l$.

 In what follows, we use the generalized binomial coefficient ${x\choose k}=\frac{1}{k!}\prod_{i=0}^{k-1}(x-i)$. Observe that ${x\choose k}$ is increasing in the  real variable $x\geq k$. Elementary computations show the inequality 
  \begin{equation*}
    \frac{{\frac{r}{l+1}\choose k}k!(l+1)^k}{{\frac{r}{l}\choose k}k!l^k}=\prod_{i=0}^{k-1}\frac{r-i(l+1)}{r-il}\leq 1.
  \end{equation*} 
This gives that the summands appearing in~\eqref{Dario} are  non-increasing in $l$ and hence they can be estimated with $l=2$. We deduce
  \begin{align*}
    |\mathcal{R}_\gamma|
    &\leq \sum_{l=2}^{r-1}\sum_{k=1}^{\lfloor\frac{r}{l}\rfloor}\binom{\lfloor\frac{r}{l}\rfloor}{k}k!l^k
    \leq \sum_{l=2}^{r-1}\sum_{k=1}^{\lfloor\frac{r}{l}\rfloor}\binom{\frac{r}{l}}{k}k!l^k
    \leq \sum_{l=2}^{r-1}\sum_{k=1}^{\lfloor\frac{r}{l}\rfloor}\binom{\frac{r}{2}}{k}k!2^k.
  \end{align*}
Furthermore, an easy computation shows that (for $0\leq k \leq x$)
$\binom{x}{k+1}-\binom{x}{k}\geq0$ if and only if $k<\frac{x}{2}$. Thus we can estimate generalized binomial coefficients with an ``almost central binomial coefficient'':
$\binom{\frac{r}{2}}{k}\leq \binom{\frac{r}{2}}{\lfloor\frac{r}{4}\rfloor}$. Thus
  \begin{align*}
    |\mathcal{R}_\gamma|
    &\leq \sum_{l=2}^{r-1}\sum_{k=1}^{\lfloor\frac{r}{2}\rfloor}\binom{\frac{r}{2}}{\lfloor\frac{r}{4}\rfloor}k!2^k 
    \leq \sum_{l=2}^{r-1}\sum_{k=1}^{\lfloor\frac{r}{2}\rfloor}\binom{\frac{r}{2}}{\lfloor\frac{r}{4}\rfloor}\left\lfloor \frac{r}{2}\right\rfloor!2^{\lfloor \frac{r}{2}\rfloor}\\
&    \leq (r-1)\left\lfloor \frac{r}{2}\right\rfloor 2^{\frac{r}{2}}\left\lfloor\frac{r}{2}\right\rfloor! 2^{\lfloor \frac{r}{2}\rfloor}
\le
(r-1)\left\lfloor \frac{r}{2}\right\rfloor\left\lfloor\frac{r}{2}\right\rfloor! 2^{r}
.\qedhere
  \end{align*}
\end{proof}

\begin{proof}[Proof of Theorem~$\ref{thm:main}$]
Notice that there are $(r-2)!$ Cayley maps with $S=R\setminus \{e\}$ (this is just the number of cyclic orderings $\rr$), the total number of Cayley maps must be greater than that, so combining Lemmas~\ref{lemma2} and~\ref{lemma:compatible mappings}, we deduce that the fraction of Cayley maps on $R$ admitting a group of automorphisms larger than $R$ is less than 
\begin{equation*}
  \frac{((r-1)\frac{r}{2} \left\lfloor r/2\right \rfloor! 2^{r})(2^{7(\log_2 r)^2+12\log_2r})}{(r-2)!},
\end{equation*} 
which goes to 0 when $r\to\infty$.
\end{proof}


\end{document}